\documentclass[11pt, reqno]{amsart}
\usepackage[utf8]{inputenc}
\usepackage{amsmath,bm}
\usepackage{amsthm}
\usepackage{amsfonts}
\usepackage{MnSymbol}
\usepackage{blindtext,amsmath,amsthm,amsfonts, textcomp, mathrsfs, mathtools}
\usepackage[shortlabels]{enumitem}

\usepackage[top=30mm,bottom=30mm,left=30mm,right=30mm]{geometry}
\newtheorem{theorem}{Theorem}[section]
\newtheorem{conjecture}{Conjecture}

\newtheorem*{theorem*}{Theorem}
\newtheorem*{remark*}{Remark}
\newtheorem*{problem*}{Problem}
\newtheorem*{conjecture*}{Conjecture}
\newtheorem*{question*}{Question}
\newtheorem{lemma}[theorem]{Lemma}
\newtheorem{proposition}[theorem]{Proposition}

\newcommand{\rom}[1]{\uppercase\expandafter{\romannumeral #1\relax}}

\begin{document}

\title[On Euler-Kronecker constants and the generalized Brauer-Siegel conjecture]{On Euler-Kronecker constants and the generalized Brauer-Siegel conjecture}

\author[Anup B. Dixit]{Anup B. Dixit\textsuperscript{1}}
\address{Department of Mathematics and Statistics\\ Queen's University\\ Jeffery Hall, 48 University Ave,\\ Kingston,\\ Canada, ON\\ K7L 3N8}
\email{anup.dixit@queensu.ca}

\date{}

\begin{abstract}
As a natural generalization of the Euler-Mascheroni constant $\gamma$, Ihara \cite{Ihara1} introduced the Euler-Kronecker constant $\gamma_K$ attached to any number field $K$. In this paper, we prove that a certain bound on $\gamma_K$ in a tower of number fields $\mathcal{K}$ implies the generalized Brauer-Siegel conjecture for $\mathcal{K}$ as formulated by Tsfasman and Vl\v{a}du\c{t}. Moreover, we use known bounds on $\gamma_K$ for cyclotomic fields to obtain a finer estimate for the number of zeros of the Dedekind zeta-function $\zeta_K(s)$ in the critical strip.
\end{abstract}



\footnotetext[1]{Research of the author was supported by a Coleman Postdoctoral Fellowship at Queen's University.}

\maketitle

\section{\bf Introduction}
\bigskip

The Euler-Mascheroni constant denoted by $\gamma$ is defined as
\begin{equation*}
    \gamma := \lim_{x\to\infty} \left( \sum_{n\leq x} \frac{1}{n} - \log x\right). 
\end{equation*}
This constant $\gamma$ appears in many areas of mathematics. For instance, it is given by the constant term in the Laurent expansion of the Riemann zeta-function,
\begin{equation}\label{zeta-gamma}
    \zeta(s) = \frac{1}{(s-1)} + \gamma + O(s-1).
\end{equation}

Motivated by \eqref{zeta-gamma},  Ihara \cite{Ihara1} introduced a generalization of $\gamma$ to any number field $K$, using the Dedekind zeta-function $\zeta_K(s)$. The Dedekind zeta-function $\zeta_K(s)$ associated to a number field $K$ is defined on the half-plane $\Re(s)>1$ as
\begin{equation*}
    \zeta_K(s) := \sum_{\mathfrak{a} \subset \mathcal{O}_K} \frac{1}{N\mathfrak{a}^s},
\end{equation*}
where $\mathfrak{a}$ runs over all non-zero integral ideals of the ring of integers $\mathcal{O}_K$. The function $\zeta_K(s)$ has an analytic continuation to the whole complex plane except for a simple pole at $s=1$. Thus, the Laurent expansion of $\zeta_K(s)$ near $s=1$ is of the form
\begin{equation*}
    \zeta_K(s) = \frac{c_{-1}}{(s-1)} + c_0 + c_1 (s-1) + \cdots
\end{equation*}
The Euler-Kronecker constant associated to $K$ is defined as
\begin{equation*}
    \gamma_K := \frac{c_{0}}{c_{-1}}.
\end{equation*}

One could also view $\gamma_K$ as the constant term in the logarithmic derivative of $\zeta_K(s)$ at $s=1$, i.e.,
\begin{equation}\label{log-derivative-gamma_K}
    \frac{\zeta_K'}{\zeta_K} (s) = \frac{-1}{(s-1)} +\gamma_K + O(s-1).
\end{equation}

In \cite{Ihara1}, Ihara established the following bounds for $\gamma_K$:
\begin{align}\label{gamma_K-bounds-ihara}
&\gamma_K \leq 2\log\log \sqrt{|d_K|} \hspace{4mm}\text{ (under GRH)}\\
&\gamma_K \geq -\log \sqrt{|d_K|} \hspace{4mm}\text{ (unconditionally)}\nonumber,
\end{align}
where $d_K$ denotes the discriminant of $K$ over $\mathbb{Q}$. Asymptotic bounds on $\gamma_K/\log\sqrt{|d_K}$ were obtained for certain families of number fields by Tsfasman in \cite{Tsfasman} and Zykin in \cite{Zykin1}. 
\medskip

In this paper, we study connections of $\gamma_K$ to two classical problems. The first one is the Brauer-Siegel conjecture, which is a statement about the rate at which the class number times the regulator, $h_K R_K$, vary in a family of number fields. In Section 2, we show that the generalized Brauer-Siegel conjecture is true for a tower of number fields if $|\gamma_K|$ satisfy certain upper bounds in the tower. These bounds are much weaker than what is expected from \eqref{gamma_K-bounds-ihara}. We also establish unconditional upper bounds on $|\gamma_K|$ for almost normal number fields and for those which have a solvable group as the Galois group of its Galois closure. The precise statements are given in Section 2.\\

In Section 3, we prove some results related to the number of zeros of $\zeta_K(s)$ in the critical strip. Denote by $N_K(T)$, the number of zeros of $\zeta_K(s)$ in the region $0<\Re(s)<1$ and $|\Im(s)|<T$. Then, it is known that for $T>2$ ,
\begin{equation*}
    N_K(T) = \frac{T}{\pi} \log \left(|d_K| \left(\frac{T}{2\pi e}\right)^{n_K} \right) + O(n_K\log  T) +O(\log |d_K|),
\end{equation*}
with the implied constants being absolute. Here $n_K$ denotes the degree and $d_K$ the discriminant of $K/\mathbb{Q}$. For a fixed large $T$, we vary $K$ in a family of cyclotomic fields and are interested in the $O(\log |d_K|)$ term in the error. In fact, using known bounds on $\gamma_K$ for almost all cyclotomic fields, we give finer results to the error terms in $N_K(T)$. Although these estimates are weaker than the known estimates for $N_K(T)$ (see Trudgian \cite{Trud}), this illustrates a new approach connecting them to bounds on $\gamma_K$.

\section{\bf The generalized Brauer-Siegel conjecture}
\bigskip

Let $K$ be a number field. Denote by $h_K$ the class number of $K$, $d_K$ the discriminant of $K$ over $\mathbb{Q}$ and $R_K$ the regulator of $K$. It is an important theme in number theory to understand how $h_K$ varies on varying $K$. Suppose $\mathcal{K} = \{K_i\}_{i\in\mathbb{N}}$ is a sequence of number fields. We call $\mathcal{K}$ to be a family if $K_i\neq K_j$ for $i\neq j$. Moreover, we call  $\mathcal{K}$ to be a tower if $K_i\subsetneq K_{i+1}$ for all $i$. A result of Heilbronn \cite{Heil}, which was earlier conjectured by Gauss, states that in a family of imaginary quadratic fields, the class number $h_K$ must tend to infinity. However, the same phenomena is not expected to hold for any general family of number fields. For instance, it is still unknown whether there are infinitely many real quadratic fields with class number $1$, although it is widely believed to be true. One of the difficulties in bounding class number is that it is difficult to isolate it from the regulator of the number field. This was observed by Siegel \cite{Sie} in 1935. He showed that for a family of quadratic fields $\{K_i\}$, the class number times the regulator  $h_{K_i} R_{K_i}$ tends to infinity as $i\to \infty$. Furthermore, he showed that
\begin{equation*}
\lim_{i\to\infty} \frac{\log h_{K_i} R_{K_i}}{\log\sqrt{|d_{K_i}|}} = 1,
\end{equation*}
for a family of quadratic fields $\mathcal{K} = \{K_i\}_{i\in\mathbb{N}}$. Since quadratic fields are determined by their discriminant (more generally, Minkowski's theorem implies that there are finitely many number fields with bounded discriminant), Siegel's result provides a rate at which $h_K R_K$ goes to infinity. Brauer \cite{Brauer} generalized this result to families of number fields, that are Galois over $\mathbb{Q}$. This is known as the classical Brauer-Siegel theorem. More precisely, he showed the following.
\begin{theorem*}[Brauer]
Let $\{K_i\}$ be a family of number fields such that $K_i/\mathbb{Q}$ is Galois for all $i$. Denote by $n_{K_i}$ the degree $[K_i: \mathbb{Q}]$. If 
\begin{equation*}
    \lim_{i\to\infty} |d_{K_i}|^{1/n_{K_i}} = \infty,
\end{equation*} 
then 
\begin{equation}\label{brauer-siegel}
\lim_{i\to\infty} \frac{\log h_{K_i} R_{K_i}}{\log\sqrt{|d_{K_i}|}} = 1.
\end{equation}
Moreover, the condition $K_i/\mathbb{Q}$ being Galois can be dropped under the assumption of generalized Riemann hypothesis (GRH). 
\end{theorem*}

The reason $h_K R_K$ appears in the above result is because of the class number formula. Recall the Dirichlet class number formula, which states that if $\rho_K$ denotes the residue of the Dedekind zeta-function $\zeta_K(s)$ at $s=1$, then
\begin{equation}\label{class-number formula}
    \rho_K = \frac{2^{r_1} (2\pi)^{r_2} h_KR_K}{\omega_K \sqrt{|d_K|}},
\end{equation}
where $r_1$ and $r_2$ denote the number of real and complex embeddings of $K$, and $\omega_K$ denotes the number of roots of unity in $K$. Using the class number formula, it is easy to see that the equation \eqref{brauer-siegel} is equivalent to
\begin{equation}\label{brauer-siegel-2}
    \lim_{i\to\infty} \frac{\log \rho_{K_i}}{\log \sqrt{|d_{K_i}|}} =0.
\end{equation}

In 2002, Tsfasman and Vl\v{a}du\c{t} \cite{TV} initiated a more extensive study of the above theorem for families of number fields, where the condition $|d_{K_i}|^{1/{n_{K_i}}} \to \infty$ can be weakened. This led to the formulation of the generalized Brauer-Siegel conjecture in \cite{TV}.\\

Define the genus of $K$ as
\begin{equation*}
g_K := \log \sqrt{|d_K|}.
\end{equation*}
Let $N_q(K)$ denote the number of non-archimedian places $v$ of $K$ such that $Norm(v)=q$. Suppose $\mathcal{K} = \{K_i\}_{i\in\mathbb{N}}$ is a family of number fields. Define the following limits.
\begin{equation*}
\phi_q := \lim_{i\to\infty} \frac{N_q(K_i)}{g_{K_i}}
\end{equation*}
for a prime power $q$.
Also define
\begin{equation*}
\phi_{\mathbb{R}} := \lim_{i\to \infty} \frac{r_1(K_i)}{g_{K_i}}, \hspace{5mm} \phi_{\mathbb{C}} := \lim_{i\to \infty} \frac{r_2(K_i)}{g_{K_i}},
\end{equation*}
where $r_1(K_i)$ and $r_2(K_i)$ are the number of real and complex embeddings of $K_i$ respectively.\\

We say that a family $\mathcal{K} = \{K_i\}$ is \textit{asymptotically exact} if the limits $ \phi_{\mathbb{R}}$, $\phi_{\mathbb{C}}$ and $\phi_q$ exist for all prime powers $q$. We say that an asymptotically exact family $\mathcal{K} = \{K_i\}$ is \textit{asymptotically bad}, if $ \phi_{\mathbb{R}} = \phi_{\mathbb{C}} = \phi_q =0$ for all prime powers $q$. This is analogous to saying that the root discriminant $ |d_{K_i}|^{1/n_{K_i}}$ tends to infinity as $i\to \infty$. If an asymptotically exact family $\mathcal{K}$ is not asymptotically bad, we say that it is \textit{asymptotically good}.
For a number field $K/\mathbb{Q}$, the Dedekind zeta-function has the Euler product
\begin{equation*}
\zeta_K(s):= \prod_{\mathfrak{P}\subset K} \left( 1- N\mathfrak{P}^{-s}\right)^{-1},
\end{equation*}
for $\Re(s)>1$, where $\mathfrak{P}$ runs over all non-zero prime ideals in the ring of integers of $K$. This can be re-written as
\begin{equation*}
    \zeta_K(s) = \prod_q \left(1- q^{-s}\right)^{-N_q(K)},
\end{equation*}
for $\Re(s)>1$, where $q$ runs over all prime powers.\\

Define the Brauer-Siegel limits (as in \cite{TV}) as follows.
For an asymptotically exact family $\mathcal{K}=\{K_i\}$, 
\begin{equation*}
BS(\mathcal{K}) := \lim_{i\to\infty} \frac{\log h_{K_i} R_{K_i}}{g_{K_i}},
\end{equation*}
\begin{equation*}
\rho(\mathcal{K}) := \lim_{i\to\infty} \frac{\log \rho_{K_i}}{g_{K_i}}.
\end{equation*}
The existence of the above limits is not clear in general. However, under GRH, the limits $BS(\mathcal{K})$ and $\rho(\mathcal{K})$ exist for any asymptotically exact family $\mathcal{K}$. The generalized Brauer-Siegel conjecture, as formulated by Tsfasman-Vl\v{a}du\c{t} \cite{TV} is stated below. 
\begin{conjecture}[Tsfasman-Vl\v{a}du\c{t}]\label{BS}
For any asymptotically exact family $\mathcal{K}$,
\begin{equation}\label{BS1}
BS(\mathcal{K})=1 + \sum_q \phi_q \log\frac{q}{q-1} -\phi_{\mathbb{R}} \log 2 - \phi_{\mathbb{C}}\log 2\pi.
\end{equation}
Using the class number formula, the above statement is equivalent to
\begin{equation}\label{BS2}
\rho(\mathcal{K}) =  \sum_q \phi_q \log\frac{q}{q-1}.
\end{equation}
\end{conjecture}

In the rest of the paper, we shall call the above conjecture as the GBS conjecture. Note that the GBS conjecture for asymptotically bad families is equivalent to the classical Brauer-Siegel conjecture. In \cite{TV}, Tsfasman-Vl\v{a}du\c{t} proved GBS for any asymptotically exact family $\mathcal{K}$ under the assumption of GRH. Unconditionally, they proved it for asymptotically good tower of almost normal number fields. Later in 2005, Zykin \cite{Zyk} showed GBS for asymptotically bad family of almost normal number fields. In \cite{Dix}, the author proved GBS unconditionally for asymptotically good towers and asymptotically bad families of number fields with solvable Galois closure. All other cases are open. For an overview of the recent results and the conjectures, the reader may refer to the excellent survey by P. Lebacque and A. Zykin \cite{Lebacque}. Furthermore, the asymptotic properties of curves over finite fields has been studied in \cite{TV2} and \cite{Tsfasman2}.

\subsection{\bf Bounds on $\gamma_K$ and the GBS conjecture}
\bigskip

In this section, we first give unconditional upper bounds on $\gamma_K$ in some cases. A number field $K$ is said to be \textit{almost normal} if there exists a tower of number fields 
\begin{equation*}
K= K_n \supset K_{n-1}\supset\cdots \supset K_1 = \mathbb{Q},
\end{equation*}
such that $K_{i+1}/K_i$ is Galois for all $1\leq i<n$.\\

\begin{theorem} \label{almost-normal-bound}
Let $K$ be an almost normal number field, not containing any quadratic sub-fields. Then
\begin{equation*}
    |\gamma_K| \leq c \,(\log |d_K|)^4 \,  n_K^3,
\end{equation*}
where $c$ is an absolute positive constant.
\end{theorem}

Let $K/\mathbb{Q}$ be a number field and $L\supseteq K \supseteq \mathbb{Q}$ be the normal closure of $K$ over $\mathbb{Q}$. We say that $K$ has \textit{solvable normal closure} if the Galois group $Gal(L/\mathbb{Q})$ is solvable.

\begin{theorem}\label{solvable-closure-bound}
Let $K$ be a number field with solvable normal closure, not containing any quadratic sub-fields. Then
\begin{equation*}
    |\gamma_K| \leq c_1 \, (\log |d_K|)^{c_2 \log\log |d_K|},
\end{equation*}
where $c_1, c_2$ are absolute positive constants.
\end{theorem}

It is important to point out that the bounds above are much weaker than the conditional bounds under GRH given by \eqref{gamma_K-bounds-ihara}. However, it is possible to utilize these weak bounds to prove the GBS conjecture for towers of such number fields. More generally, we prove the following.

\begin{theorem}\label{euler-kronecker-brauer-siegel}
Let $\mathcal{K} = \{K_i\}$ be a tower of number fields, satisfying
\begin{equation}\label{bound-condition-gamma_K}
     |\gamma_{K_i}| \ll \exp\left((\log \log |d_{K_i}|)^m\right),
\end{equation}
for an arbitrary large $m$. Then the GBS conjecture holds for $\mathcal{K}$.
\end{theorem}
In fact, in Theorem \ref{euler-kronecker-brauer-siegel}, condition \eqref{bound-condition-gamma_K} can be replaced by 
\begin{equation*}
    |\gamma_{K_i}| \ll \exp (\alpha_i),
\end{equation*}
where $\alpha_i = o\bigg( \frac{g_{K_i}}{\log g_{K_i}}\bigg)$, that is,
\begin{equation*}
    \lim_{i\to\infty}\,\, \frac{\alpha_i \log g_{K_i}}{g_{K_i}} =0.
\end{equation*}

\subsection{\bf Preliminaries}
\bigskip

In this section, we state some facts and results, which will be useful in the proof of the above theorems.

\subsubsection{\bf Exceptional zeros near $s=1$}
\medskip

For a number field $K$, $\zeta_K(s)$ has at most one real zero $\beta$ in the region
\begin{equation}\label{Siegel-zero}
    1-\frac{1}{4\log |d_K|} < \beta < 1.
\end{equation}
This zero, if it exists, is called the exceptional zero or sometimes the Siegel zero of $\zeta_K(s)$.\\ 

In \cite{Stk}, H. M. Stark showed that for an almost normal number field $K$, if $\zeta_K(s)$ has a real zero $\beta$ in the region
\begin{equation}\label{Stark-zero}
     1-\frac{1}{16\log |d_K|} < \beta < 1,
\end{equation}
then there exists a sub-field $N\subset K$, with $[N:\mathbb{Q}]=2$ such that $\zeta_N(\beta)=0$. In other words, every Stark zero must arise from a quadratic field.
\medskip

Building on the ideas of Stark and using some beautiful group theoretic techniques, V. K. Murty \cite{Km} obtained a similar result for number fields with solvable normal closure. More precisely, he showed that if $K$ has solvable normal closure over $\mathbb{Q}$ and if $\zeta_K(s)$ has a real zero $\beta$ in the region
\begin{equation}\label{Murty-zero}
1 - \frac{c}{n^{e(n)} \delta(n) \log d_K} \leq \beta < 1,
\end{equation}
then there is a quadratic field $N\subseteq K$, such that $\zeta_N(\beta) = 0$. Here, $n$ denotes the degree $[K:\mathbb{Q}]$, $c$ is an absolute positive constant,
\begin{align*}
e(n) & := \max_{p^\alpha || n} \alpha, \\
\delta(n) & := (e(n) + 1)^2 \, \, 3^{1/3} \, \, 12^{(e(n) -1)}.
\end{align*}
The above mentioned result of Stark and Murty will be crucial in the proof of Theorem \ref{almost-normal-bound} and Theorem \ref{solvable-closure-bound}.
\medskip

\subsubsection{\bf Lagarias-Odlyzko bounds}
\medskip

For a number field $K$, write
\begin{equation*}
    \zeta_K(s) = \frac{\rho_K}{(s-1)} F_K(s),
\end{equation*}
where $F_K(s)$ is entire. Define
\begin{equation}\label{log-derivative-zeta}
Z_K(s) := -\frac{1}{s-1} -\frac{d}{ds} (\log \zeta_K(s)) .
\end{equation}
From \eqref{log-derivative-gamma_K}, we have
\begin{equation*}
    \lim_{s\to 1}\, Z_K(s) = -\gamma_K.
\end{equation*}
Using Mellin transform of the Chebyshev step function, we have
\begin{equation}\label{L-O}
 \frac{Z_K(s)}{s} = - \, \frac{1}{s \, g_K} + \int_{1}^{\infty} (G_K(x) - x) \, x^{-s-1} \, dx ,
\end{equation}
for $\Re(s) > 1$, where
\begin{equation*}
G_K(x) := \sum_{\substack{q,m>1 \\ q^m \leq x}} N_q(K) \log q.
\end{equation*}
The unconditional Lagarias-Odlyzko \cite{Lag} estimate for $G_K(x)$ gives
\begin{equation*}
\bigg|G_K(x) - x \bigg| \leq C_1 \, x \exp \left( -C_2 \sqrt{\frac{\log x}{n}}\right) + \frac{x^\beta}{\beta}
\end{equation*}
for $\log x \geq C_3 \, n \, g_K^2$, where $C_1$, $C_2$, $C_3$ are positive absolute constants. Here, $\beta$ is the possible Siegel zero of $\zeta_K(s)$.
\medskip

\subsubsection{\bf Towers are asymptotically exact}
We use the following lemma, which also appears in \cite{TV}. The proof is included for sake of completeness.
\begin{lemma}[Tsfasman-Vl\v{a}du\c{t}]\label{tower-exact} 
Any infinite tower $\mathcal{K}=\{K_i\}$ is an asymptotically
exact family.
\end{lemma}
\begin{proof}
Let $L\subseteq K$. For any place $v$ of $K$, which decomposes into a set of places $\{v_1,v_2,\cdots\}$ in $L$, we have
\begin{equation*}
    \prod_i Norm(v_i) \leq (Norm(v))^{[L:K]}.
\end{equation*}
Therefore,
\begin{equation*}
    \sum_{m=1}^n m N_{p^m}(L) \leq [L:K] \sum_{m=1}^n m N_{p^m}(K).
\end{equation*}
Thus, for a tower $\{K_i\}$ and for any fixed $n$,
\begin{equation*}
    \sum_{m=1}^n \frac{m N_{p^m}(K_i)}{g(K_i)},
\end{equation*}
for $i=1,2,\cdots$ is a non-increasing sequence and hence has a limit. For $n=1$, we get the existence of $\phi_p$, $n=2$ yields the existences of $\phi_{p^2}$ and inductively we see that $\phi_{p^k}$ exists for all $k$. For archimedean places, note that if $L\subseteq K$, then
\begin{equation*}
    \frac{r_1(L)}{g(L)} + 2\frac{r_2(L)}{g(L)} \leq \frac{r_1(K)}{g(K)} + 2\frac{r_2(K)}{g(K)}.
\end{equation*}
By a similar argument as above, we conclude that $\phi_{\mathbb{R}}$ and $\phi_{\mathbb{C}}$ exists. 
\end{proof}

\subsubsection{\bf A Lemma of Stark} In \cite{Stk}, Stark proved the following lemma, which we will use below.

\begin{lemma}[Stark]\label{partial-summation-stark}
Let $Z_K(s)$ be as in \eqref{log-derivative-zeta}, then $Z_K(s)$ has the following partial summation.
\begin{equation*}
    Z_K(s) = \frac{1}{s} - \sum_{\rho} \frac{1}{s-\rho} + g_K - \frac{1}{2} r_1 (\gamma + \log 4\pi) - r_2 (\gamma + \log 2\pi) +\xi_K(s),
\end{equation*}
where $\rho$ runs over all the non-trivial zeros of $\zeta_K(s)$, $r_1$ and $r_2$ denote the number of real and complex embeddings of $K$ and 
\begin{equation*}
    \xi_K(s) := -r_1 \left(\frac{1-s}{s} + \sum_{n=1}^{\infty} \left( \frac{1}{s+2n} - \frac{1}{1+2n} \right)\right) - r_2 \left(\frac{1-s}{s} + \sum_{n=1}^{\infty} \left(\frac{1}{s+n} - \frac{1}{1+n} \right)\right).
\end{equation*}
\end{lemma}

\subsection{\bf Proof of main theorems}
\bigskip

\subsubsection{\bf Proof of Theorem \ref{almost-normal-bound}} Since $K$ is almost normal and has no quadratic sub-field, it cannot have any zero in the region \eqref{Stark-zero}. Thus, if $\zeta_K(s)$ has a Siegel zero $\beta$, it must lie in the interval
\begin{equation*}
    1-\frac{1}{4\log |d_K|} < \beta < 1-\frac{1}{16\log |d_K|}.
\end{equation*}
In other words, $1- (8g_K)^{-1} < \beta < 1- (32 g_K)^{-1}$.\\ 

Hereafter $C_i$'s will denote positive absolute constants. Since, $g_K \geq c n_K$ for some absolute positive constant $c$, we have
\begin{equation*}
    \frac{x^{\beta}}{\beta} = o \left(x \exp \left( -C_2 \sqrt{\frac{\log x}{n}}\right)\right).
\end{equation*}
Hence, for $\log x > C_3 n g_K^2$, we have
\begin{equation*}
    |G_K(x) - x| \ll \, x \exp \left( -C_2 \sqrt{\frac{\log x}{n}}\right) + \frac{x^\beta}{\beta},
\end{equation*}
where the implied constant is absolute and positive. For $\log x \leq C_3 n g_K^2$, we use the trivial estimate
\begin{equation*}
    G_K(x) = \sum_q N_q(K) \log q \leq n \sum_q \log q \ll n x\log x.
\end{equation*}
Now, the integral \eqref{L-O}, evaluated at $s=1+\theta$ gives
\begin{align*}
    \left| \frac{Z(1+\theta)}{(1+\theta)} \right| &= \left|  \int_1^\infty (G_K(x) - x) \, x^{-2-\theta} \, dx \right| +O(1)\\
    &=\left|  \int_1^{\exp(C_3 n g_K^2)} (G_K(x)-x) x^{-2-\theta} dx + \int_{\exp(C_3 n g_K^2)}^{\infty} (G_K(x) - x) x^{-2-\theta} dx \right| + O(1).
\end{align*}
Here the error $O(1)$ comes from the term $1/sg_K$.\\

The first integral
\begin{align}\label{first-integral-bound}
 \left| \int_1^{\exp(C_3 n g_K^2)} (G_K(x)-x) x^{-2-\theta} dx \right| 
 &\ll \int_1^{\exp(C_3 n g_K^2)} nx^{-1-\theta} \log x dx \nonumber\\
&\leq \frac{C_3 n^2 g_K^2}{\theta} \left(1 - \exp\left(-\theta C_3 n g_K^2 \right)\right)\nonumber\\
&\ll n^3 g_K^4.
\end{align}

We now show that the second integral is bounded. By the Lagarias-Odlyzko estimate \eqref{L-O}, we have
\begin{align}\label{bound3}
\int_{\exp(C_3 n g_K^2)}^{\infty} (G_K(x) - x) x^{-2-\theta} dx & \ll \int_{\exp(C_3 n g_K^2)}^{\infty} \exp \left(-C_2 \sqrt{\frac{\log x}{n}}\right) x^{-1-\theta} dx \nonumber\\
& \ll \int_{\exp(C_3 n g_K^2)}^{\infty} \exp \left(-C_2 \sqrt{\frac{\log x}{g_K}}\right) x^{-1-\theta} dx
\end{align}
We use the change of variables
\begin{equation*}
x = y^{g_K \log y}
\end{equation*}
to get the right hand side of \eqref{bound3} as
\begin{equation}\label{bound4}
 \ll \int_{\exp(C_3 n g_K^2)}^\infty 
y^{- \theta g_K \log y -C_2 -1} \log y dy.
\end{equation}
For large $g_K$ and any fixed $\epsilon>0$, we bound $\log y \leq y^{\epsilon}$ to get \eqref{bound4} to be
\begin{equation}\label{bound}
 \int_{\exp(C_3 n g_K^2)}^\infty y^{- \theta g_K \log y -C_2 -1+\epsilon}\, dy.
\end{equation}
We further know that in the above interval,
\begin{equation*}
\log y \geq (C_3 n g_K^2).
\end{equation*}
Hence, we have \eqref{bound} is
\begin{equation}\label{bound2}
\ll \int_{\exp(C_3 n g_K^2)}^\infty y^{- \theta g_K (C_3 n g_K^2) -C_2 -1+\epsilon}\, dy \ll 1.
\end{equation}
Putting together \eqref{first-integral-bound}, \eqref{bound2} in \eqref{L-O}, we get that for $\theta\in (0,1)$
\begin{equation*}
    \left| \frac{Z(1+\theta)}{(1+\theta)} \right| \ll n^3 g_K^4.
\end{equation*}
Thus by \eqref{log-derivative-gamma_K}, we get Theorem \ref{almost-normal-bound}.

\medskip

\subsubsection{\bf Proof of Theorem \ref{solvable-closure-bound}} The proof here follows along the same lines as in the proof of Theorem \ref{almost-normal-bound}. Since $K$ has solvable normal closure over $\mathbb{Q}$ with no quadratic sub-fields, if $\zeta_K(s)$ has a Siegel zero $\beta$, by \eqref{Murty-zero} it must lie in the region
\begin{equation*}
    1-\frac{1}{4\log |d_K|}< \beta < 1 - \frac{c}{n^{e(n)} \delta(n) \log |d_K|}.
\end{equation*}

Incorporating this into the proof of Theorem \ref{almost-normal-bound}, using the Lagarias-Odlyzko bounds \eqref{L-O}, we get the required result (for more details see the proof of Lemma 2.5 in \cite{Dix}).

\medskip

\subsubsection{\bf Proof of Theorem \ref{euler-kronecker-brauer-siegel}} Let $K$ be a number field. Write
\begin{equation*}
\zeta_K (s)= \frac{\rho_K}{(s-1)} F_K(s).
\end{equation*}
Taking $\log$ on both sides and dividing by $g_K$, we get for $s = 1+\theta_K$
\begin{equation}\label{main}
\frac{\log \zeta_K (1+\theta_K)}{g_K} = \frac{\log \rho_K}{g_K} + \frac{\log F_K(1+ \theta_K)}{g_K} - \frac{\log \theta_K}{g_K}.
\end{equation}
\\

For a family of number fields $\mathcal{K}=\{K_i\}$, in order to prove GBS, it suffices find a sequence of $\theta_{K_i} \to 0$ such that as $i\to\infty$,
\begin{align}\label{brauer-siegel-limits}
    &\frac{\log \zeta_{K_i}(1+\theta_{K_i})}{g_{K_i}} \to \sum_q \phi(q) \log \left(\frac{q}{q-1}\right),\\
    &\frac{\log F_{K_i}(1+ \theta_{K_i})}{g_{K_i}} \to 0, \nonumber \\
    &\frac{\log \theta_{K_i}}{g_{K_i}} \to 0\nonumber.
\end{align}

The difficulty lies in the choice of $\theta_{K_i}$. The convergence in \eqref{brauer-siegel-limits} may not be uniform and hence does not allow for interchanging summation and limits for any choice of $\theta_{K_i}$'s. This is precisely the reason why we get the unconditional results only for towers of number fields, and not for asymptotically exact families in general. In case of towers, it is possible to utilize the monotone convergence theorem to overcome the issue. Moreover, the choice of $\theta_{K_i}$ cannot be too small, which would result in $\log \theta_{K_i}/g_{K_i}$ not approaching $0$.\\

In \cite{TV}, it is shown that for any asymptotically exact family of number fields,
\begin{equation*}
\limsup_{i\to\infty} \frac{\log \rho_{K_i}}{g_{K_i}} \leq \sum_q \phi_q \log \frac{q}{q-1}.
\end{equation*}
Thus, to prove the Theorem \ref{euler-kronecker-brauer-siegel}, first note by Lemma \ref{tower-exact} that any tower of number fields is asymptotically exact. Hence, it suffices to show that for some choice of $\theta_{K_i} \to 0$,

\begin{equation}\label{zeta}
\liminf_{i\to \infty} \frac{\zeta_{K_i}(1+\theta_{K_i})}{g_{K_i}} \geq \sum_q \phi_q \log \frac{q}{q-1},
\end{equation}
\begin{equation}\label{F(s)}
\limsup_{i\to\infty} \frac{\log F_{K_i}(1+ \theta_{K_i})}{g_{K_i}} \leq 0,
\end{equation}
and
\begin{equation}\label{theta}
\lim_{i\to\infty} \frac{\log \theta_{K_i}}{g_{K_i}} = 0.
\end{equation}

We first show that \eqref{F(s)} is implied by a certain choice of $\theta_{K_i}$'s under the assumption of bounds on $|\gamma_{K_i}|$. Recall that
\begin{equation*}
    \lim_{s\to 1} Z_K(s) = -\gamma_K.
\end{equation*}
We show that for $\theta < 1/n_K$,
\begin{equation}\label{stark-lemma-use}
    \big|Z_K(1+\theta) - \lim_{s\to 1} Z_K(s) \big| = O(1).
\end{equation}
To see this, we use Stark's lemma $\ref{partial-summation-stark}$, which gives
\begin{equation*}
    Z_K(1+\theta) - \lim_{s\to 1} Z_K(s) = \sum_{\rho} \left(\frac{1}{1+\theta -\rho} -\frac{1}{(1-\rho)}\right ) + \lim_{s\to 1} \bigg( \xi_K(s+\theta) - \xi_K(s) \bigg) + O(1).
\end{equation*}
From the definition of $\xi_K(s)$, it is easy to see that
\begin{equation*}
    \lim_{s\to 1} \xi_K(s) =0.
\end{equation*}
Moreover, for $\theta<1/n_K$,
\begin{align*}
    \xi_K(1+\theta) &= -r_1 \left( -\frac{\theta}{1+\theta} + \sum_{n=1}^{\infty}\left( \frac{1}{(1+2n + \theta)} - \frac{1}{(1+2n)}\right)\right) \\
    & \hspace{2cm} -r_2 \left( -\frac{\theta}{1+\theta} + \sum_{n=1}^{\infty}\left( \frac{1}{(1+n + \theta)} - \frac{1}{(1+n)}\right)\right) \ll \theta \, n_K \ll 1.
\end{align*}
Therefore, to show \eqref{stark-lemma-use}, it suffices to show that
\begin{equation*}
    \sum_{\rho} \frac{1}{1+\theta- \rho} - \frac{1}{1-\rho} \ll 1.
\end{equation*}
By the functional equation of $\zeta_K(s)$, we know that if $\rho$ is a non-trivial zero of $\zeta_K$, then so is $1-\rho$. Therefore, we have
\begin{equation*}
    \sum_{\rho} \frac{1}{1+\theta- \rho} - \frac{1}{1-\rho} = \sum_{\rho} \frac{1}{\theta + \rho} - \frac{1}{\rho}.
\end{equation*}
Clubbing together $\rho$ and $\overline{\rho}$ from the summation, we write
\begin{align}\label{sum-zero-squares}
    \sum_{\rho} \frac{1}{\theta + \rho} - \frac{1}{\rho} &= \frac{1}{2}\sum_{\rho} \left( \frac{1}{(\theta+\rho)} + \frac{1}{\overline{\theta+\rho}} - \frac{1}{\rho} - \frac{1}{\overline{\rho}}\right)\nonumber\\
    &= \frac{1}{2} \sum_{\rho} \left( \frac{2\theta + 2\Re(\rho)}{|\theta+\rho|^2} - \frac{2\Re(\rho)}{|\rho|^2}\right)\nonumber\\
    &= \frac{1}{2} \sum_{\rho} \frac{2\theta |\rho|^2 - 2 \Re(\rho) |\theta|^2 - 4 \theta \Re(\rho)^2}{|\rho|^2 |\theta+ \rho|^2}\nonumber\\
    &\ll \theta \sum_{\rho} \frac{1}{|\theta+\rho|^2}.
\end{align}

To estimate \eqref{sum-zero-squares}, we use the upper bounds on the number of zeros of $\zeta_K(s)$ given by Jensen's theorem. Let $N_K(T)$ denote the number of zeros of $\zeta_K(s)$ in the region $0<\Re(s)<1$ and $|\Im(s)|<T$. Using Jensen's theorem, one can see that
\begin{equation*}
    |N_K(T+1) - N_K(T)| \ll n_K \log T,
\end{equation*}
where the implied constant is absolute. The detailed computation of the above for a more general case can be found in (\cite{Dix2}, Lemma 4.1.4).\\

Thus, by partial summation, we get that
\begin{equation*}
    \theta \sum_{\rho} \frac{1}{|\theta + \rho|^2} \ll \theta\, n_K \sum_{n=1}^{\infty} \frac{\log n}{n^2}  \ll 1,
\end{equation*}
since $\theta <1/n_K$. This proves \eqref{stark-lemma-use}. Therefore, for a choice of $\theta_K < 1/n_K$, 
\begin{equation*}
    |Z_K(1+\theta_K)| \ll |\gamma_K|.
\end{equation*}
Now, we have
\begin{equation}\label{equation1}
    \frac{\log F_K(1+\theta_K)}{g_K} = \int_{0}^{\theta_K} Z_K(1+\theta) \, d\theta \ll \frac{\theta_K\, |\gamma_K|}{g_K}. 
\end{equation}
If for some $m>0$,
\begin{equation*}
    |\gamma_{K_i}| \ll \exp\left( (\log g_{K_i})^m \right)
\end{equation*}
for all $i$, we choose 
\begin{equation*}
    \theta_K = \exp \left(-(\log g_K)^{m+2}\right).
\end{equation*} 
From \eqref{equation1}, as $i\to\infty$,
\begin{equation*}
    \frac{\log F_{K_i}(1+ \theta_{K_i})}{g_{K_i}} \to 0 \text{ and }\frac{\log \theta_{K_i}}{g_{K_i}} \to 0.
\end{equation*}

Now, we are left to show \eqref{zeta}. Note that
\begin{align*}
\frac{\zeta_{K_i}(1+\theta)}{g_{K_i}}& = \sum_q \frac{N_q(K_i)}{g_{K_i}} \log \frac{1}{1-q^{-1-\theta}}\\
&= \sum_p  \frac{N_p(K_i)}{g_{K_i}} \log \frac{1}{1-p^{-1-\theta}} + \sum_{\substack{p\text{ prime}, k>1 \\ q=p^k}} \frac{N_q(K_i)}{g_{K_i}} \log \frac{1}{1-q^{-1-\theta}}.
\end{align*}
If $\mathcal{K}=\{K_i\}$ is a tower, we know that $\phi_p \leq \frac{N_p(K_i)}{g_{K_i}}$. Therefore,
\begin{equation*}
\sum_p  \frac{N_p(K_i)}{g_{K_i}} \log \frac{1}{1-p^{-1-\theta}} \geq \sum_p \phi_p  \log \frac{1}{1-p^{-1-\theta}},
\end{equation*}
for any $\theta>0$. We also have
\begin{equation*}
\sum_{\substack{p\text{ prime}, k>1, \\ q = p^k}} \bigg( \frac{N_q(K_i)}{g_{K_i}} \log \frac{1}{1-q^{-1-\theta}} \bigg) \longrightarrow \sum_{\substack{p\text{ prime}, k>1, \\ q = p^k}} \phi_q \log \frac{1}{1-q^{-1-\theta}} 
\end{equation*}
uniformly for $\theta > -\delta$, for some $\delta>0$. Hence, we get
\begin{equation*}
\liminf_{i\to \infty} \zeta_{K_i}(1+\theta_{K_i}) \geq \sum_q \phi_q \log \frac{q}{q-1}.
\end{equation*}
This proves GBS for towers of number fields. 

\section{\bf On the number of zeros of $\zeta_K(s)$ for cyclotomic fields}
\bigskip

Let $\zeta_K(s)$ be the Dedekind zeta function associated to the number field $K/\mathbb{Q}$. It satisfies a functional equation of the form
\begin{equation*}
    \Lambda_K(s)\zeta_K(s) = \Lambda_K(1-s) \zeta_K(1-s),
\end{equation*}
where $\Lambda_K(s)$, is given as
\begin{equation*}
    \Lambda_K(s) = |d_K|^{s/2} (\pi^{-s/2} \Gamma(s/2))^{r_1} (2(2\pi)^{-s} \Gamma(s))^{r_2}. 
\end{equation*}

By the above functional equation, it is easy to see that $\zeta_K$ has zeros in the region $\Re(s)<0$ coming from the poles of the $\Gamma$-function at negative integers. These are called the trivial zeros. Moreover, because of the Euler-product, $\zeta_K(s)$ does not have any zeros on $\Re(s)>1$. The symmetry of the functional equation implies that all the zeros of $\zeta_K$ in the region $\Re(s)<0$ are in fact trivial. Therefore, all the non-trivial zeros of $\zeta_K$ lie in the critical strip $0<\Re(s)<1$.\\

Define
\begin{equation*}
    N_K(T) := \#\bigg\{s : \zeta_K(s)=0, 0<\Re(s)<1, |\Im(s)|<T \bigg\},
\end{equation*}
which counts the number of zeros in the critical strip up to height $T$, according to multiplicities. Using Riemann-von Mongoldt-type formula, it can be shown that for $T>2$
\begin{equation}\label{numberofzeros}
    N_K(T) = \frac{T}{\pi} \log \left(|d_K| \left(\frac{T}{2\pi e}\right)^{n_K} \right) + O\left(\log \left(|d_K| \, T^{n_k}\right) \right),
\end{equation}
where the implied constant is absolute.\\

Suppose, we fix a large $T$, and vary $K$ over a family. Then, we are interested in the implied constant associated to the error term $O(\log |d_K|)$ in \eqref{numberofzeros}. In this direction, a result of H. Kadiri and N. Ng (see \cite{Habiba}) sheds some light. An improvement of their techniques leads to the following result due to T. Trudgian \cite{Trud}, which is perhaps the best known result so far. He showed that for $T>1$
\begin{equation}\label{trudgian-bound}
    \left|N_K(T) - \frac{T}{\pi} \log \left(|d_K| \left(\frac{T}{2\pi e}\right)^{n_K} \right)\right| \leq 0.317 (\log |d_K| + n_K \log T) + 6.333 n_K + 3.482.
\end{equation}

In certain cases, one could produce even better asymptotic results. For instance, if we consider an asymptotically bad family $\mathcal{K}=\{K_i\}$ of number fields, and fix a very large $T$, then in \cite[Table 2]{Trud} yields
\begin{equation}\label{trudgian-consequence}
     \left|N_{K_i}(T) - \frac{T}{\pi} \log \left(|d_{K_i}| \left(\frac{T}{2\pi e}\right)^{n_{K_i}} \right)\right| = 0.248 (\log |d_{K_i}|) + o(\log |d_{K_i}|) +O(n_{K_i} \log T),
\end{equation}
where the implied constant in the $O$-term is absolute and the $o$-notation bounds the growth of the function as $i\to\infty$.\\

Let $K=\mathbb{Q}(\zeta_p)$ be the cyclotomic field where $p$ is a prime and $\zeta_p$ denotes the primitive $p$-th root of unity. Then, from \eqref{numberofzeros}, we have
\begin{equation*}
    N_K(T) = \frac{(p-1)}{\pi} T\log T + \left(\frac{(p-2)\log p - (p-1)\log (2\pi e)}{\pi}\right) T + O\left((p-2)\log p + (p-1) \log T \right),
\end{equation*}
where the implied constant is absolute, independent of $p$.\\

Our goal is to understand the implied constant of the $O(p\log p)$-term in the error, upon varying $p$. In this section, we will show that certain known bounds on the Euler-Kronecker constants quite easily produce bounds on this implied constant. We note that these bounds are not better than what we already have from \eqref{trudgian-consequence}. However, it is worth appreciating the connection of $\gamma_K$ and this problem, especially the simple argument which leads to these bounds. \\

Let $\gamma_p$ denote the Euler-Kronecker constant associated with $K=\mathbb{Q}(\zeta_p)$ with $p$ prime. In \cite{Ihara1}, Ihara conjectured that $\gamma_p >0$ for all primes $p$. The basis for this conjecture was perhaps the observation that in order for $\gamma_p$ to be negative, there must be a large number of small primes $l$ which split completely in $\mathbb{Q}(\zeta_p)$. But, the conjecture is known to be false (see \cite{Ford}), with an explicit counterexample 
\begin{equation*}
    \gamma_{964477901} = -0.182\cdots <0.
\end{equation*}

It was also shown by Ford, Luca and Moree \cite{Ford} that if the Hardy-Littlewood $k$-tuple conjecture is true, then $\gamma_p <0$ infinitely often. Nevertheless, such a phenomena would occur rarely.\\

For our purpose, we will use some unconditional results due to V. K. Murty and M. Mourtada \cite{Km1}, who showed that for almost all primes $p$,
\begin{equation}\label{Kumar-Mourtada}
    1\geq \frac{\gamma_p}{\log p} > -11.
\end{equation}

It is also interesting to note that (see \cite{Ford}), assuming Hardy-Littlewood and Elliot-Halberstam conjecture, for almost all primes $p$, we have
\begin{equation*}
    1>\frac{\gamma_p}{\log p} >1-\epsilon.
\end{equation*}

Let 
\begin{equation*}
    Q := \left\{ p \text{ prime } : 1 \geq \frac{\gamma_p}{\log p} > -11 \right\}.
\end{equation*}

By \eqref{Kumar-Mourtada}, $Q$ consists of almost all primes and for $p\in Q$, we get
\begin{equation*}
    |\gamma_p| \leq 11  \log p.
\end{equation*}

\begin{proposition}\label{main-theorem-number-zeros}
    Let $\mathcal{K}=\{K_i\}$ be a family such that $K_i = \mathbb{Q}(\zeta_{p_i})$, where $p_i\in Q$. Then, assuming GRH, for a large fixed $T$, we have
    \begin{equation*}
        \left(N_{K_i}(T) - \frac{T}{\pi} \log \left(|d_{K_i}| \left(\frac{T}{2\pi e}\right)^{n_{K_i}} \right)\right) = c (p_i-2)\log p_i + o(p_i\log p_i) +O((p_i-1) \log T),
    \end{equation*}
    with the constant $c$ satisfying
    \begin{equation*}
        -\frac{4}{\pi} \leq c \leq \frac{1}{\pi} \left(2\tan^{-1}(2) -\frac{4}{5}\right)
    \end{equation*}
\end{proposition}

Here, the assumption of GRH is not a restriction and one can produce similar results without assuming GRH with more careful analysis. However, we assume it to make the computations easier. \\

\subsection{Proof of Proposition \ref{main-theorem-number-zeros}}
From Stark's Lemma \ref{partial-summation-stark} and \eqref{log-derivative-gamma_K}, we have
\begin{equation}\label{euler-kronecker-summation-formula}
     \sum_{\rho}\frac{1}{\rho} =  \gamma_p + \frac{(p-2)}{2} \log p - \frac{(p-1)}{2} (\log 2\pi + \gamma) +1,
\end{equation}
where $\rho$ runs over all the non-trivial zeros of $\zeta_K$.\\

By the functional equation of $\zeta_K$, if $\rho$ is a zero, then so is $1-\rho$. Assuming GRH, we get
\begin{align}\label{1/rho}
    \sum_{\rho} \frac{1}{\rho} &= \frac{1}{2} \sum_{\Im(\rho)=t} \frac{1}{1/2+it} + \frac{1}{1/2-it}\\
    &= \frac{1}{2}\sum_{\Im(\rho)=t} \frac{1}{1/4 + t^2}. \notag
\end{align}

Note that
\begin{equation*}
    \lim_{M\to\infty} \sum_{n=1}^{M} \frac{N_K(n)-N_K(n-1)}{1/4 + n^2} \leq \sum_{\Im(\rho)=t} \frac{1}{1/4 + t^2} \leq \lim_{M\to\infty} \sum_{n=0}^{M} \frac{N_K(n+1)-N_K(n)}{1/4 + n^2}.
\end{equation*}

Using partial summation, we have
\begin{equation*}
    \lim_{M\to\infty} \sum_{n=0}^{M} \frac{N_K(n+1)-N_K(n)}{1/4 + n^2} = \frac{N_K(M+1)}{1/4 + M^2} + \int_0^M \frac{2u}{(1/4 + u^2)^2} N_K(u+1)\, du.
\end{equation*}

For large $M$,
\begin{equation*}
    \frac{N_K(M+1)}{1/4 + M^2} \to 0,
\end{equation*}
because $N_K(T) \ll T\log T$. If $M$ is large, and $K=K_i$, using \eqref{numberofzeros} we have
\begin{align*}
    \int_0^M \frac{2u}{(1/4 + u^2)^2} N_{K_i}(u+1)\, du &= \left(\frac{1}{\pi} \int_0^M \frac{2u(u+1)}{(1/4 + u^2)^2}\, du  \right) (p_i-2)\log p_i + c (p_i-2)\log p_i \\
    &+  O\left((p_i-1) T\log T\right) + o(p_i\log p_i).
\end{align*}

Using
\begin{equation*}
    \frac{1}{\pi}\lim_{M\to\infty} \int_0^M \frac{2u^2}{(1/4 + u^2)^2}\, du =1,
\end{equation*}
 and 

\begin{equation*}
 \frac{1}{\pi}\lim_{M\to\infty} \int_0^M \frac{2u}{(1/4 + u^2)^2}\, du =\frac{4}{\pi},
\end{equation*}

we get
\begin{equation*}
    \int_0^M \frac{2u}{(1/4 + u^2)^2} N_{K_i}(u+1)\, du = \left( 1+\frac{4}{\pi} +c \right) (p_i-2)\log p_i +  O\left((p_i-1) T\log T\right) + o(p_i\log p_i).
\end{equation*}

Comparing this with \eqref{euler-kronecker-summation-formula}, we get
\begin{equation*}
    c \geq -\frac{4}{\pi}.
\end{equation*}

By a similar argument and using
\begin{equation*}
    \frac{1}{\pi}\lim_{M\to\infty} \int_1^M \frac{2u^2}{(1/4 + u^2)^2}\, du = \frac{1}{\pi} \left( \frac{4}{5} +\pi - 2\tan^{-1}(2)\right),
\end{equation*}
we get
\begin{equation*}
    c \leq \frac{1}{\pi} \left(2\tan^{-1}(2) -\frac{4}{5}\right).
\end{equation*}
This proves Proposition \ref{main-theorem-number-zeros}.\\

To obtain analogous result without the assumption of GRH, one should follow a similar argument as above, by replacing \eqref{1/rho} with $\sum 1/\rho = 1/2 (\sum 1/\rho + \sum 1/\overline{\rho})$. 

\bigskip
\bigskip

\section{\bf Concluding Remarks}
From Stark's lemma \ref{partial-summation-stark}, we have for any number field $K$,
\begin{equation*}
    \sum_{\rho} \frac{1}{\rho} = \gamma_K + \frac{1}{2} \log |d_K| - \frac{1}{2} r_1 (\gamma + \log 4\pi) - r_2 (\gamma + \log 2\pi) +1.
\end{equation*}

The sum $\sum_{\rho} 1/\rho$ can be interpreted in terms of the Li coefficient. Recall that the Li's coefficients are defined for $n\geq 1$ as
\begin{equation*}
    \lambda_n = \sum_{\rho} \left(1- \left(1-\frac{1}{\rho}\right)^n\right).
\end{equation*}
Li's criterion asserts that the Riemann hypothesis is true if and only if $\lambda_n$ is positive for all $n$. It is clear that
\begin{equation*}
    \lambda_1 = \sum_{\rho} \frac{1}{\rho}.
\end{equation*}
Thus, $\gamma_K$ also holds the information on the positivity of $\lambda_1$. Moreover, any estimate on $\gamma_K$ leads to an estimate on $\sum_{\rho} 1/\rho$. This observation could also be used to produce upper bounds for the low lying zeros of $\zeta_K(s)$.

\section*{Acknowledgments}
I thank Prof. M. R. Murty and Prof. V. K. Murty for helpful comments on an earlier version of this paper. I would also like to thank Prof. T. Trudgian for pointing out the best bound in equation \eqref{trudgian-bound} for asymptotically bad families. I thank both the anonymous referees for detailed comments.

\end{document}